\documentclass[11pt]{amsart}

\usepackage{hyperref} 
\usepackage[dvips]{graphicx}
\usepackage[T1]{fontenc}
\usepackage{mathptmx}
\usepackage{amssymb}
\usepackage{color}
\usepackage[titletoc]{appendix}

\newcommand{\bbar}{\begin{pmatrix}}
\newcommand{\ebar}{\end{pmatrix}}

\newcommand{\bdm}{\begin{displaymath}}
\newcommand{\edm}{\end{displaymath}}
\newcommand{\beq}{\begin{equation}}
\newcommand{\beqa}{\begin{eqnarray}}
\newcommand{\beqas}{\begin{eqnarray*}}
\newcommand{\eeq}{\end{equation}}
\newcommand{\eeqa}{\end{eqnarray}}
\newcommand{\eeqas}{\end{eqnarray*}}
\newcommand{\dd}{\textup{d}}

\newcommand{\C}{{\mathbb C}}
\newcommand{\D}{{\mathbb D}}

\newcommand{\real}{{\mathbb R}}
\newcommand{\SSS}{{\mathbb S}}

   \newtheorem{theorem}{Theorem}

   \newtheorem{lemma}[theorem]{Lemma}

 \theoremstyle{remark}

\begin{document}

\title{Remarks on the boundary curve of a constant mean curvature topological disc}
\author{David Brander}
\address{Department of Applied Mathematics and Computer Science\\ Matematiktorvet, Building 303 B\\
Technical University of Denmark\\
DK-2800 Kgs. Lyngby\\ Denmark}
\email{dbra@dtu.dk}
\author{Rafael Lop\'ez}
\address{Departamento de Geometr\'ia  y Topolog\'ia \\ Universidad de Granada \\
 18071 Granada, Spain}
\email{rcamino@ugr.es}

\begin{abstract}
We discuss some consequences of the existence of the holomorphic quadratic Hopf differential on 
a conformally immersed constant mean curvature topological disc with analytic boundary.
In particular, we derive a formula for the mean curvature as a weighted average of the normal
curvature of the boundary curve, and a condition for the surface to be totally umbilic in terms
of the normal curvature. 
 \end{abstract}
\maketitle
\section{Introduction}
A constant mean curvature (CMC) surface is a  model for a fluid interface.  As such, a basic
configuration of interest is that of a CMC topological disc with boundary. It is well known that any CMC surface
admits a real analytic conformal parameterization.
We will also assume that the boundary is a real analytic curve in $\real^3$;  under reasonable regularity 
assumptions \cite{muller2002} one can then show that the surface extends analytically across the boundary.
Applying the smooth Riemann mapping theorem, we thus
 consider a conformal CMC immersion $f: U \supset \overline \D\to \real^3$, where $\overline \D$ is the closed unit disc,
and $U$ an open subset of $\C$.

Given a simple closed curve $C$ in $\real^3$,
Hildebrandt \cite{hildebrandt1970} showed that there exists at least one (possibly branched)
CMC topological disc spanning $C$, if the
mean curvature $H$ is not too large: if $C$ lies in a ball of radius $R$ then solutions exist provided that $|H| \leq 1/R$.

In this note we record some observations concerning the consequences that the existence of the holomorphic 
quadratic differential of Hopf \cite{hopf1951, Hopfbook}  has  for the relationship between the boundary curve
$\gamma := f \big|_{\partial \D}$ and the surface.  For a general real analytic 
space curve, the choice of a surface normal $N$ along the curve
and a constant $H$ locally determine a unique CMC $H$ surface.  However, if the curve is closed and the surface is required
to be a regular topological disc, then we will see that the normal curvature $\kappa_n$ and the geodesic torsion $\tau_g$
(obtained from the derivative of $N$ along the curve) are
closely related and must satisfy some integral relations, amongst them:
\beqa
 \label{integral1}
\int_0^{2\pi} (H- \kappa_n(t))\left(\frac{\dd s}{\dd t}\right)^2 \dd t = 0   \\
\int_0^{2\pi}  \tau_g(t)\left(\frac{\dd s}{\dd t}\right)^2 \dd t  = 0,  \label{integral2}
\eeqa
where $z=e^{it}$ is the restriction of the conformal parameterization to $\partial \D$ and $s$ is the
arc-length with respect to the metric induced from $\real^3$. 
The first equation gives an expression for $H$ as an average of the
normal curvature with respect to the measure $\mu= (\dd s/ \dd t)^2 \dd t$:
\[
H = \frac{1}{M} \int_0^{2\pi} \kappa_n(t)  \left(\frac{\dd s}{\dd t}\right)^2 \dd t, \quad \quad
      M := \int_0^{2\pi} \left( \frac{\dd s}{\dd t}\right)^2 \dd t.
\]
It is important to note that these formulae are only valid for a parameterization $t$ such that $z=e^{it}$ extends 
to the disc as a conformal coordinate.

 It obviously follows from \eqref{integral1} that the constant $H$ lies in the range
$[\min(\kappa_n), \max(\kappa_n)]$.  We will in fact conclude:
\begin{theorem}  \label{mainthm}
Let $f: \D \cup \partial \D \to \real^3$ be a CMC $H$ immersion of a topological disc with regular analytic boundary 
$C=f(\partial \D)$.  Let $\kappa_n$ and $\tau_g$ be the normal curvature and geodesic torsion of the curve
$f|_{\partial \D}$. 
\begin{enumerate}
\item  
  The surface is totally umbilic if and only if at least one of  $\kappa_n$ or $\tau_g$ is constant. 
In this case, both are constant and $\kappa_n \equiv H$ and $\tau_g \equiv 0$.
	\item
Otherwise,  the mean curvature $H$ must satisfy
\[
\min_{\partial \D} (\kappa_n)  <  H < \max_{\partial \D} (\kappa_n).
\]
\end{enumerate}
\end{theorem}
Any homotopically trivial simple closed 
curve on a CMC surface can be taken as
$\partial \D$ in  Theorem \ref{mainthm}.  This
places restrictions on the geometry of such curves.
For example,   there are none with constant non-zero geodesic torsion;
  and  none that are lines of curvature ($\tau_g=0$)  unless the surface is totally umbilic.

 If the surface is totally umbilic,  i.e.,~ part of a $2$-sphere or a plane, the geodesic torsion 
along any curve at all is zero, as is the quantity $\kappa_n-H$.
But the characterization of total umbilicity as the vanishing of
  just \emph{one} of these quantities along a closed curve is not
valid if the curve does not bound a topological disc - counterexamples are easily given on
a round cylinder.

Theorem \ref{mainthm} is proved in Section \ref{section2}.  In Section \ref{section3} we
briefly discuss applications to CMC surfaces with circular boundary and to capillary surfaces.

\section{The  Hopf differential and some consequences}  \label{section2}
\subsection{The Hopf differential along a line}  
Let $U \subset \C$ be an open set
and $f: U \to \real^3$ be a conformally parameterized 
immersion,  i.e.,
\[
2 \langle f_z , f_{\bar z} \rangle =  \langle f_x, f_x \rangle = \langle f_y, f_y \rangle = 
   \nu^2 >0,
\]
where $\nu : U \to \real_{>0}$.  Given a choice $N$ of unit normal, the mean curvature is
\[
H= \frac{1}{2} \nu^{-2} \langle f_{xx}+ f_{yy}, N \rangle.
\]
As observed by Hopf \cite{Hopfbook}, if we define 
$Q = \langle N, f_{zz} \rangle$,
then the Codazzi equation gives $Q_{\bar z} = \nu^2 H_{z} /2$, so the mean curvature is constant if and only if
$Q$ is holomorphic.  Assuming this,  the \emph{Hopf differential}
is the (coordinate independent) holomorphic quadratic differential  
\[
\mathcal{H} = Q(z) \dd z^2, \quad Q = \langle N, f_{zz} \rangle.
\]

Suppose that $U$ contains the real line.
We want an expression for $\mathcal{H}$ along $\real$, in terms of the geometric 
data of the curve $\gamma(x)=f(x,0)$.  
   We first assume that $x$ is the arc-length parameter along $\gamma$,
i.e., that
\[
\langle f_x(x,0), f_x (x,0) \rangle = \nu^2(x,0) = 1.
\]
This assumption can be achieved on a sufficiently small open set containing the real line
by changing to the conformal coordinates
$\zeta = \int_0^z \hat \nu(z) \dd z$, where $\hat \nu$ is the holomorphic extension of 
$\nu(x,0)$.  Under this assumption,the Darboux frame along $\gamma$ is:
\[
X(x) = f_x(x,0), \quad Y(x) = f_y(x,0), \quad N(x) = N(x,0).
\]
 The geodesic and normal curvatures, $\kappa_g$ and $\kappa_n$,  and
 geodesic torsion $\tau_g$ of $\gamma$ are defined by:
\beqas
X^\prime &=& \kappa_g Y + \kappa_n N, \\
Y^\prime &=& -\kappa_g X + \tau_g N, \\
N^\prime &=& -\kappa_n X -\tau_g Y.
\eeqas
Now the coefficient $Q$ of $\mathcal{H}$ is:
\[
Q= \frac{1}{4} \left( \langle N, f_{xx} \rangle - \langle N, f_{yy} \rangle \right)
    - \frac{1}{2}i \langle N, f_{xy} \rangle.
\]
Along $\gamma$, we have 
$\langle N, f_{xx}  \rangle = \kappa_n$ and
$\langle N,  f_{yy} \rangle  = 2H- \langle N, f_{xx} \rangle$.
Finally, $f_{xy} = Y^\prime$, so $\langle N, f_{xy} \rangle = \tau_g$. Therefore, along $y=0$,
\[
Q = \frac{1}{2} ( \kappa_n -H - i \tau_g).
	\]
Along the curve, the second fundamental form is $II=\kappa_n \dd x^2 + 2\tau_g \dd x \dd y + (2H-\kappa_n) \dd y^2$,
and the principal curvatures are $\kappa_\pm = H \pm \sqrt{(H-\kappa_n)^2+ \tau_g^2}$.
Thus the geodesic torsion is a measure of how far the curve $\gamma$ deviates from being a line of curvature,
and an umbilic point
corresponds to $\kappa_n-H=\tau_g=0$, that is, to the zeros of the Hopf differential.
	
Under a change of coordinates,  the Hopf differential transforms as 
$\mathcal{H} = 
\langle N, f_{zz} ( \dd \zeta/ \dd z )^2 + f_z \dd^2 \zeta/\dd z^2 \rangle \dd z^2
= \langle N, f_{zz}\rangle    (\dd \zeta/\dd z)^2 \dd z^2$.
Dropping the assumption that $\nu(x,0)=1$, we have:
\[
Q(x,0) = \frac{1}{2} (\kappa_n(x) - H - i  \tau_g(x)) \nu^2(x,0).
\]

\subsection{Proof of Theorem \ref{mainthm}}
Now let $f: U \supset \overline \D \to \real^3$ be a conformally parameterized CMC $H$ immersion, as in
the introduction.
Let $\log$ be any branch of the logarithm function,
and set $z = -i \log(w)$.  Then, on an annulus containing $\SSS^1$, we have the situation of the
previous subsection, where the functions $\kappa_n(x)$, $\tau_g(x)$ and $\nu(x,0)$ are periodic
with period $2 \pi$.  Since $\dd z^2 = -(1/w^2) \dd w^2$, we have, in the $w$ coordinate,
\[
\mathcal{H} = -\frac{1}{2 w^2}\left(\alpha(w) + i \beta(w) \right) \dd w^2 = \tilde Q(w) \dd w^2,
\]
where,  writing $\hat \kappa_n$, $\hat \tau_g$ and $\hat \nu$
 for the holomorphic extensions of  $\kappa_n(x)$, $\tau_g(x)$ and $\nu(x,0)$
respectively,
\[
\alpha(w)=(\hat \kappa_n(w) - H)  \hat \nu^2(w) \quad \hbox{and} \quad
\beta(w) = - \hat \tau_g(w)  \hat \nu^2(w)
\]
 are  holomorphic on a neighbourhood of $\SSS^1$,
 real-valued along $\SSS^1$, and $\hat \nu$ is non-vanishing along $\SSS^1$.
 Because $\mathcal{H}=\tilde Q(w) \dd w^2$
 is holomorphic on $\D$, it follows that  the function 
\[
g(w) = \alpha(w) + i \beta(w) = -2 w^2 \tilde Q (w)
\]
 vanishes at least to second order at $w=0$.  
This means that the integrals along $\SSS^1$ of the functions $g(w)$, $g(w)/w$ and
$g(w)/w^2$ are all zero, which gives us:
\beqas
\int_0^{2\pi} \alpha(e^{it}) \dd t =0 =  \int_0^{2\pi} \beta(e^{it}) \dd t , \\
\int_0^{2\pi} \alpha(e^{it}) e^{\pm i t} \dd t = 0=\int_0^{2\pi} \beta(e^{it}) e^{\pm i t} \dd t .
\eeqas
The first line amounts to  \eqref{integral1}  and \eqref{integral2},
and the second line gives analogous identities with the integrands multiplied by $\cos(t)$ and by $\sin(t)$.
The fact that $\alpha$ and $\beta$ are real on $\SSS^1$, together with the vanishing conditions above, gives us the Fourier expansions:
\bdm
\alpha(w) = \sum_{n=2}^\infty (\bar \alpha _n w^{-n} + \alpha_n w^n) \quad \hbox{and} \quad
\beta(w) = \sum_{n=2}^\infty (i \bar \alpha _n w^{-n} -i  \alpha_n w^n),
\edm
which precisely encodes the relationship between 
the quantities $\nu$, $\tau_g$ and $\kappa_n-H$ that follow from the holomorphicity
of the Hopf differential on $\D$.  From this one can read off, for example, that $\tau_g$ is constant
if and only if $\kappa_n-H$ is constant, and in this case both are zero.
 Since the vanishing of the Hopf differential characterizes umbilics, this proves the first part of
 Theorem \ref{mainthm}. For the second part, \eqref{integral1} implies that
$\min(\kappa_n) \leq H \leq \max(\kappa_n)$. Equality cannot occur, because then the vanishing of
the integral \eqref{integral1} of $H-\kappa_n$ would imply that $\kappa_n \equiv H$, i.e., that the 
surface is totally umbilic.

\section{Some applications of Theorem \ref{mainthm}}  \label{section3}
\subsection{CMC discs bounded by a circle}
Given a closed curve in $\real^3$, Brezis and Coron \cite{breziscoron1984} showed that there are at least two 
geometrically distinct CMC $H$ discs spanning the curve, provided $H$ is small enough (see also \cite{steffen1986}).
However, it is not known precisely how many there are.
 The most basic case imaginable is where the boundary is a 
circle, which we may as well take to be of radius $1$. If $0<|H| <1$ there are obviously at least two CMC $H$ discs
bounded by this circle, a large and a small spherical cap on a sphere of radius $1/|H|$.
 It is a long standing open question, going back as far as the 1980's \cite{koiso1986, ebmr1991} whether or not there are more solutions in the case $0<|H| < 1$. 
 A survey of this problem can be found in \cite{lopezbook}.
From Theorem \ref{mainthm} we have that the unit sphere is the only solution for $|H|=1$, and,
for solutions that are not totally umbilic, 
\[
-1 \leq \min_{\partial \D} (\kappa_n)  < H < \max_{\partial \D} (\kappa_n)  \leq 1.
\]
Brito and Sa Earp \cite{britosaearp} essentially  showed this using the   \emph{balancing formula} or \emph{flux formula},
which was found by R. Kusner in his Ph.D. thesis,  published in
\cite{KKS}, and a proof given for the immersed case in \cite{lopezmontiel}. It can be stated as follows:
if $f: M \to \real^3$ is a CMC $H$ immersion of a compact surface with boundary, and
$\gamma(s)$ the arc-length parameterization of the boundary, then
\[
\int_{\partial \Sigma} \langle Y, a \rangle \dd s + H \int_{\partial \Sigma} \langle \gamma \times \gamma^\prime, a \rangle \dd s=0,
\]
where $Y$ is the unit conormal and $\gamma$ a unit speed parameterization  of the boundary curve,
and $a$ is any constant vector in $\real^3$.  The flux formula can be proved (\cite{lopezmontiel}) by observing that the 
$1$-form $(Hf + N) \times \dd f$ is closed if $f$ has constant mean curvature $H$.
This formula differs from the integral formula \eqref{integral1}, 
as it does not depend on the topology of $\Sigma$ and does not relate to
 any conformal structure. 

 For the special case that $\partial \Sigma$ is a circle $\gamma(s) = (\cos s, \sin s,0)$, taking $a=(0,0,1)$
allows one to again obtain the above bounds on $H$:
the surface normal along $\gamma$ is $N= -\kappa_n \gamma + \kappa_g  a$ and 
the conormal is $Y= -\kappa_g \gamma - \kappa_n a $.  So the flux formula becomes
\[
 \int_{\partial \Sigma} ( \kappa_n  - H  )  \frac{\dd s}{\dd t} \dd t=0.
\]
This is not the same as \eqref{integral1} (and note that the integral in \eqref{integral1}  has
  a priori a different value for every choice of conformal coordinate),
 but it does  give the same upper and lower bounds on $H$ for the case of circular boundary.

Note that the holomorphicity of the Hopf differential is  used in \cite{alp1999} to prove that 
a \emph{stable} CMC topological disc with circular boundar must be a spherical cap, and
in \cite{palmer2003} to
give a condition for the instability of a higher genus CMC surface with circular boundary.

\subsection{Application to capillary surfaces}
If $M$ is a CMC surface with boundary $\partial M$, and the boundary curve $\gamma$ lies in another surface $S$ (the support surface), such that the
surfaces intersect with constant contact angle along  $\gamma$, then $M$ is called a \emph{capillary surface}.  If $M$ is a topological disc, then this is
a model for the surface of a drop of liquid resting on a solid surface in the absence of gravity. 
The contact angle depends on the physical properties of the solid (hydrophilic/hydrophobic).
Two examples are shown to the right in Figure \ref{examplesfig}, each constructed by adding a ruled
support surface with a prescribed contact angle to the boundary of a CMC surface.

It is known that if the support surface is part of a plane or a sphere then a capillary surface with
disc topology is necessarily totally umibilic \cite{nitsche1985, finnmccuan, choe}.
This follows directly from Theorem \ref{mainthm}, because any curve in a sphere or plane is a line
of curvature (zero geodesic torsion) and the  Terquem-Joachimsthal theorem \cite{spivak3} states that if the intersection of two surfaces $M_1$ and $M_2$ is a line of curvature on $M_1$,
then it is also a line of curvature on $M_2$ if and only if the surfaces intersect at constant angle. 

\begin{figure}[ht]
\centering
\begin{small}
$
\begin{array}{ccc}
\includegraphics[height=25mm]{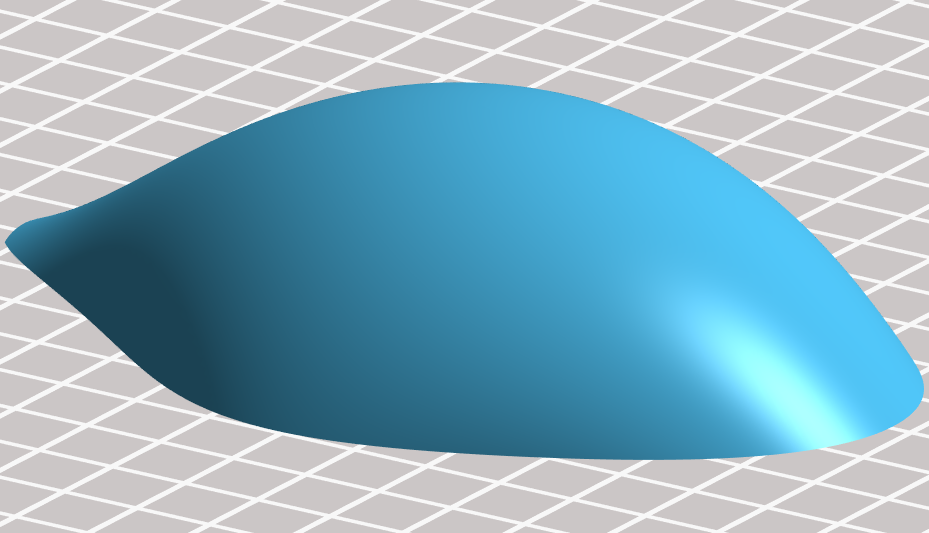}  \,\, & \,\,
\includegraphics[height=25mm]{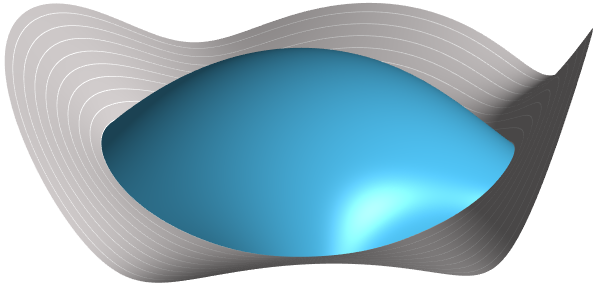}   \,\, & \,\,
\includegraphics[height=25mm]{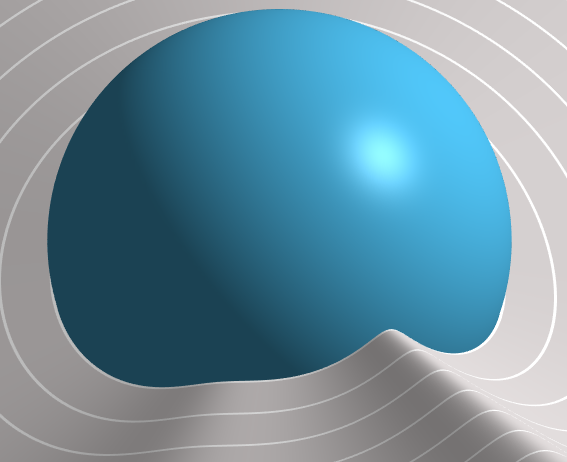}   
\end{array}
$
\end{small}
\caption{Left:  this CMC surface cannot meet a plane with
constant contact angle because it is not a spherical cap. Middle, right:
capillary surfaces with contact angle $\pi/2$.  The surface to the right is part of a sphere,
so the intersection curve  is necessarily a line of curvature in the support surface. } \label{examplesfig}
\end{figure}   

To state the consequence just mentioned in full generality, one can extend the  Terquem-Joachimsthal theorem  
to curves that are not necessarily lines of curvature:
\begin{lemma}
Suppose that a surface $M_1$ meets a surface $M_2$  along a regular curve $\gamma$. Denote by $\tau_{g, j}$ the geodesic
torsion of $\gamma$ with respect to the surface $M_j$. Then the surfaces meet at constant angle if and only 
if $\tau_{g,1} = \tau_{g,2}$.
\end{lemma}
\begin{proof}
Assume $\gamma$ is parameterized by arc-length. Let $N_j$ and $Y_j = N_j \times \gamma^\prime$ denote the unit normal and co-normal along
 $\gamma$ with respect to the surface $M_j$.  Then $\langle Y_1, N_2 \rangle = \sin(\theta) = - \langle Y_2, N_1 \rangle$,
where $\theta$ is the angle from $N_1$ to $N_2$. Thus,
\beqas
\langle N_1, N_2 \rangle ^\prime &=&  -\tau_{g,1} \langle Y_1, N_2 \rangle - \tau_{g,2} \langle Y_2, N_1 \rangle \\
&=& (\tau_{g,1} - \tau_{g,2}) \langle Y_2, N_1 \rangle.
\eeqas
If $\theta$ is constant, i.e., $\langle N_1, N_2 \rangle^\prime \equiv 0$, then  the product $(\tau_{g,1} - \tau_{g,2}) \langle Y_2, N_1 \rangle$ is
zero. Then either $\tau_{g,1} \equiv \tau_{g,2}$ or  $\langle Y_2, N_1 \rangle$ vanishes at some point; but if
the latter occurs then $\langle Y_2, N_1 \rangle=-\sin \theta \equiv 0$, because $\theta$ is constant, 
and hence $N_1 \equiv  \pm N_2$, and $\tau_{g,1} \equiv \tau_{g,2}$.
 The converse is also clear.
\end{proof}
Together with Theorem \ref{mainthm}, this implies
\begin{theorem}  \label{capillarythm}
Let $M$ be a closed topological disc, immersed as a capillary surface with support surface $S$ 
and analytic boundary.  Let
$\gamma$ denote the boundary curve of $M$. Then $M$ is totally umbilic if and only if $\gamma$ has constant
geodesic torsion in the support surface $S$.  In this case, the geodesic torsion is zero.
\end{theorem}


\end{document}